\documentclass{amsart}
\usepackage{amssymb,stmaryrd}
\usepackage{enumerate}
\usepackage[shortlabels]{enumitem}

\newtheorem{thm}{Theorem}[section]
\newtheorem{lemma}[thm]{Lemma}

\newtheorem{cor}[thm]{Corollary}
\newtheorem{fact}[thm]{Fact}
\newtheorem{factt}{Fact}
\newtheorem{claim}{Claim}[thm]
\newtheorem{mainthm}{Theorem}

\theoremstyle{definition}
\newtheorem{q}[thm]{Question}
\newtheorem{defn}[thm]{Definition}
\newtheorem{defi}[factt]{Definition}
\newtheorem{remark}[thm]{Remark}

\DeclareMathOperator{\CG}{CG}
\DeclareMathOperator{\ad}{AD}
\DeclareMathOperator{\cf}{cf}
\DeclareMathOperator{\ns}{NS}
\DeclareMathOperator{\im}{Im}
\DeclareMathOperator{\reg}{Reg}
\DeclareMathOperator{\nacc}{nacc}
\DeclareMathOperator{\acc}{acc}
\DeclareMathOperator{\dom}{dom}
\DeclareMathOperator{\otp}{otp}
\DeclareMathOperator{\ps}{Ps}
\DeclareMathOperator{\onto}{{\sf onto}}

\newcommand\axiomfont[1]{\textsf{\textup{#1}}}
\newcommand\zfc{\axiomfont{ZFC}}
\newcommand\gch{\axiomfont{GCH}}
\newcommand\sch{\axiomfont{SCH}}
\newcommand\s{\subseteq}
\newcommand\br{\blacktriangleright}
\renewcommand\mid{\mathrel{|}\allowbreak}
\renewcommand\restriction{\mathbin\upharpoonright}

\newenvironment{why}[1][Proof]{\proof[#1]\mbox{}}{\endproof}

\title[Ladder systems and cmc spaces]{Diamond on ladder systems and countably metacompact topological spaces}
\date{Preprint as of January 14, 2024. For the latest version, visit \textsf{http://p.assafrinot.com/63}.}

\author {Rodrigo Carvalho}
\address{Department of Mathematics, Bar-Ilan University, Ramat-Gan 5290002, Israel.}
\email{rodrigo.rey.carvalho@gmail.com}

\author{Tanmay Inamdar}
\address{Department of Mathematics, Ben-Gurion University of the Negev, P.O.B. 653, Be’er Sheva, 84105 Israel}
\email{tci.math@protonmail.com}

\author {Assaf Rinot}
\address{Department of Mathematics, Bar-Ilan University, Ramat-Gan 5290002, Israel.}
\urladdr{http://www.assafrinot.com}


\begin{document}
\begin{abstract}
The property of countable metacompactness of a topological space gets its importance from Dowker's 1951 theorem that
the product of a normal space $X$ with the unit interval $[0,1]$ is again normal iff $X$ is countably metacompact.
In a recent paper, Leiderman and Szeptycki studied $\Delta$-spaces,
which are a subclass of the class of countably metacompact spaces.
They proved that a single Cohen real introduces a ladder system $ L$ over the first uncountable cardinal
for which the corresponding space $X_{ L}$ is not a $\Delta$-space, and asked whether there is a ZFC example of a ladder system $ L$ over some cardinal $\kappa$ for which $X_{ L}$ is not countably metacompact,
in particular, not a $\Delta$-space.
We prove that an affirmative answer holds for the cardinal $\kappa=\cf(\beth_{\omega+1})$.
Assuming $\beth_\omega=\aleph_\omega$, we get an example at a much lower cardinal, namely $\kappa=2^{2^{2^{\aleph_0}}}$,
and our ladder system $L$ is moreover $\omega$-bounded.
\end{abstract}

\maketitle

\section{Introduction}

Throughout, $\kappa$ denotes a regular uncountable cardinal.
A \emph{ladder system} over a stationary subset $S$ of $\kappa$ is a sequence ${\vec L}=\langle A_\delta\mid\delta\in S\rangle$ such that each $A_\delta$ is a cofinal subset of $\delta$.
It is \emph{$\xi$-bounded} iff $\otp(A_\delta)\le\xi$ for all $\delta\in S$.
The corresponding topological space $X_{\vec L}$ has underlying set $(\kappa\times\{0\})\cup(S\times\{1\})$ with all points in $(\kappa\times\{0\})$ being isolated
and, for every $\delta\in S$, the neighborhoods of $(\delta,1)$ consisting of sets of the form $(A\times\{0\})\cup\{(\delta,1)\}$ for some $A$ an end segment of $A_\delta$.

A topological space $X$ is a \emph{$\Delta$-space} (resp.~\emph{countably metacompact}) iff
for very decreasing sequence $\langle D_n\mid n<\omega\rangle$ of subsets of $X$ (resp.~closed subsets of $X$) with empty intersection,
there is a decreasing sequence $\langle U_n\mid n<\omega\rangle$ of open subsets of $X$ with empty intersection
such that $D_n\s U_n$ for all $n<\omega$.

It is well-known that the product of two normal topological spaces need not be normal,
but what about the product of a normal space $X$ and the unit interval $[0,1]$?
It is a classical theorem of Dowker \cite{dowker} that the product $X\times[0,1]$ is again normal iff $X$ is countably metacompact,
hence the importance of this notion.
The notion of a $\Delta$-space is due to Knight \cite{MR1196219}.

In a recent paper by Leiderman and Szeptycki \cite{leiderman2023deltaspaces}, a systematic study of $\Delta$-spaces is carried out, motivated by the $C_p$-theory of such spaces (see \cite[Theorem 2.1]{MR4214339}).
Section~5 of \cite{leiderman2023deltaspaces} is dedicated to the study of spaces of the form $X_{\vec L}$. It is proved that in $\zfc$
there is an $\omega$-bounded ladder system ${\vec L}$ over $\aleph_1$ for which $X_{\vec L}$ is countably metacompact,
that under Martin's axiom all $\omega$-bounded ladder systems ${\vec L}$ over $\aleph_1$ satisfy that $X_{\vec L}$ is countably metacompact,
and that in the forcing extension after adding a single Cohen real,
there exists an $\omega$-bounded ladder system ${\vec L}$ over $\aleph_1$ for which the space $X_{\vec L}$ is not a $\Delta$-space.
At the end of that section, Problem~5.11 asks whether there is a $\zfc$ example of a ladder system ${\vec L}$ over some cardinal $\kappa$ whose corresponding space $X_{\vec L}$ is not countably metacompact,
hence not a $\Delta$-space. We answer this question in the affirmative, as follows.
\begin{mainthm}\label{thma}
For $\kappa:=\cf(\beth_{\omega+1})$ there are co-boundedly many regular cardinals $\mu<\beth_\omega$ such that $E^\kappa_\mu:=\{\delta<\kappa\mid\cf(\delta)=\mu\}$
carries a $\mu$-bounded ladder system ${\vec L}$ such that  $X_{\vec L}$ is not countably metacompact.
\end{mainthm}

Theorem~\ref{thma} fits into a well-known programme of obtaining analogues in $\zfc$ of statements which are undecidable at small cardinals. Often times, the price is that these results concern higher cardinals and it suggests the fruitfulness of an asymptotic viewpoint to statements in infinite combinatorics.
So, here the Leiderman-Szeptycki consistency result for $\kappa=\aleph_1$ is obtained in $\zfc$
at $\kappa=\cf(2^\lambda)$, where $\lambda:=\sup\{2^{\aleph_0},2^{2^{\aleph_0}},2^{2^{2^{\aleph_0}}},\ldots\}$.
The proof builds heavily on Shelah's contributions to this programme,
where he previously showed that refined forms of Jensen's results for G\"odel's constructible universe \cite{MR0309729} hold asymptotically in any universe of set theory.
This includes refined forms of the $\gch$ \cite{Sh:460}, of the square principle \cite{Sh:420} and of the diamond principle \cite{Sh:775}.
These refined results often state that a desired phenomenon holds at all but to some indispensable small set of  `bad' cardinals,
and typically, these `bad' cardinals include $\aleph_0$ (see, for instance, \cite[Theorem~3.5]{MR2078366}, \cite[\S2]{Sh:829} and \cite[Lemma~8.13]{MR2768693}, the latter two highlighting the role played by $\aleph_1$-complete ideals having well-defined rank functions).
In our context, this raises the question whether a ladder system $\vec L$ as in Theorem~\ref{thma} may be obtained to concentrate at points of countable cofinality,
thereby ensuring that the corresponding space $X_{\vec L}$ be first countable.
The next two theorems provide sufficient conditions beyond $\zfc$ for an affirmative answer.

The first theorem yields a ladder system of interest from a weak arithmetic hypothesis,
the failure of which is consistent \cite{MR1087344}, but has a very high consistency strength.\footnote{Indeed, the failure asserts that $2^\lambda\ge\lambda^{+\omega+1}$ for every infinite cardinal $\lambda$,
so it in particular requires the singular cardinals hypothesis ($\sch$) to fail everywhere.}

\begin{mainthm}\label{thmb}
If there exists an infinite cardinal $\lambda$ such that
$\kappa:=2^{2^{2^{\lambda}}}$ is a finite successor of  $\lambda$,
then $E^\kappa_\omega$ carries an $\omega$-bounded ladder system $\vec L$ for which $X_{\vec L}$ is not countably metacompact.

In particular, if $2^{2^{2^{\aleph_0}}}<\aleph_\omega$, then the conclusion holds for $\kappa:=2^{2^{2^{\aleph_0}}}$.
\end{mainthm}

The second theorem yields a ladder system of interest from the existence of a particular type of a Souslin tree.

\begin{mainthm}\label{thmc} If there exists a (resp.~coherent) $\kappa$-Souslin tree,
then there exists a ladder system $\vec L$ over some stationary subset of $\kappa$ (resp.~over $E^\kappa_\omega$) for which $X_{\vec L}$ is not countably metacompact.
\end{mainthm}

Note that in Theorems \ref{thma}, \ref{thmb}, and \ref{thmc}, the cardinal $\kappa$ was not defined in terms of the $\aleph$-hierarchy. This is not a coincidence, as a simple generalisation of \cite[Claim~1]{MR2099600} implies that for any ordinal $\alpha$, upon forcing Martin's axiom together with the continuum being greater than $\aleph_\alpha$, for every $\omega$-bounded ladder system $\vec L$ over a stationary subset of $E^\kappa_\omega$,  $X_{\vec L}$ is countably metacompact.

\medskip

Ultimately, our proofs of Theorems \ref{thma} and \ref{thmb} go through a diamond-type principle on ladder systems studied by Shelah under various names (\cite{Sh:775, Sh:829}): middle diamond, super black box, $\ps_1$.
We opt for the following nomenclature.
\begin{defi}\label{defdio} For a ladder system $\vec L=\langle A_\delta\mid\delta\in S\rangle$ over some stationary $S\s\kappa$
and a cardinal $\theta$,
$\diamondsuit(\vec L,\theta)$ asserts the existence of a sequence $\langle f_\delta\mid\delta\in S\rangle$ such that:
\begin{itemize}
\item for every $\delta\in S$, $f_\delta$ is a function from $A_\delta$ to $\theta$;
\item for every function $f:\kappa\rightarrow\theta$,
there are stationarily many $\delta\in S$ such that $f\restriction A_\delta=f_\delta$.
\end{itemize}
\end{defi}

Note that Jensen's diamond principle $\diamondsuit(S)$ is simply $\diamondsuit(\vec L,2)$ for the degenerate ladder system $\vec L=\langle \delta\mid\delta\in S\rangle$.
In \cite{Sh:775}, Shelah proved that $\diamondsuit(\vec L,\mu)$ holds in $\zfc$ for various ladder systems $\vec L$ and cardinals $\mu$.
A central case reads as follows.

\begin{factt}[Shelah]\label{sh775}
Suppose that $\Lambda\le\lambda$ is a pair of uncountable cardinals such that $\Lambda$ is a strong limit.
Denote $\kappa := \cf(2^\lambda)$.
Then, for co-boundedly many regular cardinals $\mu<\Lambda$,
there exists a $\mu$-bounded ladder system $\vec C=\langle C_\delta\mid \delta\in E^\kappa_\mu\rangle$ with each $C_\delta$ a club in $\delta$
such that $\diamondsuit(\vec C, \mu)$ holds.\footnote{A ladder system as above, i.e., consisting of sets which are closed subsets of their suprema, is called a \emph{$C$-sequence}.}
\end{factt}

That Fact~\ref{sh775} should have applications in set-theoretic topology
was anticipated ever since \cite{Sh:775} was written,
and yet Theorem~\ref{thma} is the first application.

At the beginning of this paper, we shall give an accessible proof of Fact~\ref{sh775}.
Our motivation for doing so is twofold.
First, this will pave the way for the proof of Theorem~\ref{thmb}.
Second, ladder systems are a rich source of examples and counterexamples in set-theoretic topology (see for example \cite{MR728725}),
and weak diamonds have proved to be useful in studying abstract elementary classes (see for example \cite{MR2532039}),
so we hope that this paper will help popularise this result of Shelah and its variations.
In particular, we expect the next theorem to find further applications.
\begin{mainthm}\label{thmd} Suppose that $\aleph_\omega$ is a strong limit.
For every positive integer $n$,
for all infinite cardinals $\mu\le\theta<\aleph_\omega$,
there are a cardinal $\kappa<\aleph_\omega$,
a $\mu$-bounded ladder system $\vec L=\langle A_\delta\mid\delta\in E^\kappa_\mu\rangle$
and a map $g:\kappa\rightarrow\theta$
such that for every function $f:[\kappa]^n\rightarrow\theta$, there are stationarily many $\delta\in E^\kappa_\mu$ such that
$f``[A_\delta]^n=\{g(\delta)\}$.
\end{mainthm}

\subsection{Organisation of this paper}
In Section~\ref{sec2} we include a proof of Fact~\ref{sh775} and some variants. This section is written at a slower pace hoping to introduce readers to the basic construction of diamonds on ladders systems.
The reader we have in mind here is a topologist or a model-theorist who is not necessarily familiar with Shelah's revised $\gch$ and the approachability ideal.
While the results here are due to Shelah or can be extracted from \cite{Sh:775}, some of the proofs make use of ideas and concepts from recent papers of the authors of this paper.

In Section~\ref{sec3} we prepare the ground for Theorem~\ref{thmb},
and we prove Theorem~\ref{thmd} which is of independent interest. Here, we shall assume the reader is comfortable with the content of Section~\ref{sec2}.

Section~\ref{sec1} is focused on topological applications of diamonds on ladder systems and our other main results. In particular, the proof of Theorems \ref{thma}, \ref{thmb} and \ref{thmc} will be found there.
\subsection{Notation and conventions}\label{conventions}
The set of all infinite regular cardinals below $\kappa$ is denoted by $\reg(\kappa)$.
For a set $X$, we write $[X]^{\kappa}$ for the collection of all subsets of $X$ of size $\kappa$. The collections $[X]^{\le\kappa}$ and $[X]^{<\kappa}$ are defined similarly.
In the specific case when $X$ is a set of ordinals and $n$ is an integer $\ge2$, we will identify $[X]^n$ with the set of ordered tuples $(\alpha_1,\ldots,\alpha_n)$ where $\alpha_1<\cdots<\alpha_n$ are all from $X$.
For a set of ordinals $A$, we write $\acc(A) := \{\alpha\in A \mid \sup(A \cap \alpha) = \alpha > 0\}$
and $\nacc(A) := A \setminus \acc(A)$.
For cardinals $\theta$ and $\mu$, $\theta^{+\mu}$ denotes the $\mu^{\text{th}}$ cardinal after $\theta$: so if $\theta = \aleph_\alpha$, then $\theta^{+\mu}= \aleph_{\alpha+\mu}$.
The map $\alpha\mapsto\beth_\alpha$ is defined by recursion on the class of ordinals, setting 	$\beth_0:=\aleph_0$, $\beth_{\alpha+1}:=2^{\beth_\alpha}$ and $\beth_\alpha:=\bigcup_{\beta<\alpha}\beth_\beta$ for every infinite limit ordinal $\alpha$.

\section{Diamonds on ladder systems}\label{sec2}
\subsection{Motivation} Diamonds on ladder systems have already found deep applications in algebra \cite{Sh:898,Sh:1028},
but in view of the overall limited number of applications thus far,
we would like to further motivate this concept by recalling an application to model theory and mentioning a natural strengthening of it that is still an open problem.

To start, note that Definition~\ref{defdio} concerns itself with guessing one-dimensional colourings $f: \kappa \rightarrow \theta$,
but in the same way one may also wish to guess higher-dimensional colourings $f$ with $\dom(f)=[\kappa]^2$ or even $\dom(f)={}^{<\omega}\kappa$.
Here is a concrete variation:
\begin{defn}\label{defdiinfnite} For a ladder system $\vec L=\langle A_\delta\mid\delta\in S\rangle$ over some stationary $S\s\kappa$
and a cardinal $\theta$,
$\diamondsuit^{<\omega}(\vec L,\theta)$ asserts the existence of a sequence $\langle f_\delta\mid\delta\in S\rangle$ such that:
\begin{itemize}
\item for every $\delta\in S$, $f_\delta$ is a function from ${}^{<\omega}A_\delta$ to $\theta$;
\item for every function $f:{}^{< \omega}\kappa\rightarrow\theta$,
there are stationarily many $\delta\in S$ such that $f\restriction {}^{<\omega}A_\delta=f_\delta$.
\end{itemize}
\end{defn}
One of the reasons for the interest in higher-dimensional variations of Definition~\ref{defdio} is provided by the following application.
\begin{fact}[Shelah, {\cite[Theorem 0.1]{Sh:775}}] \label{substructuremd}
Suppose that $\vec L=\langle A_\delta\mid\delta\in S\rangle$ is a ladder system over some stationary $S\s\kappa$
and $\tau$ is an infinite cardinal such that
$\diamondsuit^{<\omega}(\vec L,2^\tau)$ holds.

Then, for every relational language $\mathcal L$ of size at most $\tau$, there is a sequence $\langle M_\delta \mid \delta \in S\rangle$ of $\mathcal L$-structures such that for every $\mathcal L$-structure $M$ with carrier set $\kappa$, for stationarily-many $\delta \in S$, $M_\delta$ is a substructure of $M$ with carrier set $A_\delta$.
\end{fact}

The application is rather straightforward.
Let $\mathcal L$ be a relational language with relational symbols $\langle R_i \mid i< \tau\rangle$. The hypothesis of $\diamondsuit^{<\omega}(\vec L,2^\tau)$ provides us with a sequence $\langle f_\delta\mid\delta\in S\rangle$ where for every $\delta \in S$, $f_\delta: {}^{< \omega}A_\delta \rightarrow {}^\tau2$, and
such that for every function $f:{}^{< \omega}\kappa \rightarrow {}^\tau 2$, the set $\{\delta \in S\mid f_\delta=f\restriction {}^{< \omega}A_\delta\}$ is stationary.

For every $\delta \in S$, let $M_\delta$ be the structure with carrier set $A_\delta$ and relations given as follows: for a finite tuple $\langle \alpha_0,\ldots,\alpha_{n-1}\rangle \in {}^{< \omega}A_\delta$
and an index $i<\tau$,
$$M_\delta \models R_i(\langle \alpha_0,\ldots,\alpha_{n-1}\rangle) \iff f_\delta(\langle \alpha_0,\ldots,\alpha_{n-1}\rangle)(i) = 1 \ \&\ \mathrm{arity}(R_i) = n.$$

Now given an $\mathcal L$-structure $M= \langle \kappa, R_i \rangle_{i< \tau}$, define a function $f:{}^{< \omega}\kappa \rightarrow \mathcal {}^\tau2$ by letting
$$f(\langle \alpha_0,\ldots,\alpha_{n-1}\rangle)(i) = 1 \iff M\models R_i(\langle \alpha_0,\ldots,\alpha_{n-1}\rangle) \ \&\ \mathrm{arity}(R_i) = n..$$
Then it is clear that $M_\delta$ is a substructure of $M$ whenever $f_\delta = f\restriction {}^{< \omega}A_\delta$.
\begin{remark}
As indicated by Shelah in \cite[\S 0]{Sh:775},
obtaining diamonds on ladder systems which would allow for strengthening of Fact~\ref{substructuremd}
by requiring the $M_\delta$'s to be \emph{elementary} substructures of $M$ is one of the key open problems in this area.
\end{remark}

As for the matter of when $\diamondsuit^{<\omega}(\vec L,\theta)$ or other higher-dimensional variants hold, we leave to the interested reader the task of translating the arguments from the upcoming subsection to their purpose with the assurance that no extra ingenuity is required.
As the differences are minor we have chosen to focus on the most transparent case.

\subsection{Results}\label{subsec21}

In this subsection, we reproduce some results from \cite{Sh:775} with the goal of proving Fact~\ref{sh775}
and laying the groundwork for Section~\ref{sec3}.

We start by considering two (one-dimensional) generalisations of $\diamondsuit(\vec L,\theta)$ and a generalisation of the Devlin-Shelah weak diamond principle $\Phi$ \cite{MR469756}.
\begin{defn}
Suppose that $\vec L=\langle A_\delta\mid\delta\in S\rangle$
is a ladder system over some stationary $S\s\kappa$,
and that $\mu,\theta$ are cardinals greater than $1$.
\begin{itemize}
\item $\diamondsuit^*(\vec L,\mu,\theta)$ asserts the existence of a sequence $\langle\mathcal P_\delta\mid\delta\in S\rangle$ such that:
\begin{itemize}
\item for every $\delta\in S$, $|\mathcal P_\delta|<\mu$;
\item for every function $f:\kappa\rightarrow\theta$,
there are club many $\delta\in S$ such that $f\restriction A_\delta\in\mathcal P_\delta$.
\end{itemize}
\item $\diamondsuit(\vec L,\mu,\theta)$ asserts the existence of a sequence $\langle\mathcal P_\delta\mid\delta\in S\rangle$ such that:
\begin{itemize}
\item for every $\delta\in S$, $|\mathcal P_\delta|<\mu$;
\item for every function $f:\kappa\rightarrow\theta$,
there are stationarily many $\delta\in S$ such that $f\restriction A_\delta\in\mathcal P_\delta$.
\end{itemize}
\item $\Phi(\vec L,\mu,\theta)$ asserts that for every function $F:(\bigcup_{\delta\in S}{}^{A_\delta}\mu)\rightarrow\theta$,
there exists a function $g:S\rightarrow\theta$ such that, for every function $f:\kappa\rightarrow\mu$,
there are stationarily many $\delta\in S$ such that $F(f\restriction A_\delta)= g(\delta)$.
\end{itemize}
\end{defn}

We encourage the reader to determine the monotonicity properties of the above principles; another easy exercise is to verify that for every ladder system $\vec L$ over a subset of $\kappa$ and every cardinal $\mu$, $\diamondsuit^*(\vec L, \mu^\kappa, \mu)$ holds.

Note that for $\kappa$  a successor cardinal, the principle $\diamondsuit^*(S)$ is simply $\diamondsuit^*(\vec L,\kappa,2)$ for the degenerate ladder system $\vec L=\langle \delta\mid\delta\in S\rangle$.
Also note that $\diamondsuit^*(\vec L,\mu,\theta)\implies \diamondsuit(\vec L,\mu,\theta)$ and $\diamondsuit(\vec L,\theta)\iff\diamondsuit(\vec L,2,\theta)$.
Less immediate from these two observations, but clear after expanding the definitions is that $\diamondsuit(\vec L,\mu)\implies\Phi(\vec L,\mu,\theta)$ for any cardinal $\theta$. This implication admits a converse, as follows.

\begin{lemma}\label{l210} Suppose that
$\vec L=\langle A_\delta\mid\delta\in S\rangle$  is a $\xi$-bounded ladder system over some stationary $S\s\kappa$.
If $\Phi(\vec L,\mu,\mu^{|\xi|})$ holds, then so does $\diamondsuit( \vec L, \mu)$.
\end{lemma}
\begin{proof} Denote $\theta:=\mu^{|\xi|}$.
Let $\vec h = \langle h_\tau \mid \tau< \theta\rangle$ be some enumeration of  ${}^\xi\mu$.
For every $\delta \in S$, fix an injection $\psi_\delta: A_\delta\rightarrow\xi$.
Fix a function $F:(\bigcup_{\delta \in S}{}^{A_\delta}\mu)\rightarrow \theta$ such that for all $\delta \in S$ and $\bar f: A_\delta \rightarrow \mu$,
$$(F(\bar f) = \tau)\implies(h_\tau\circ\psi_\delta=\bar f).$$
Now, assuming that $\Phi(\vec L, \mu, \theta)$ holds, we may fix a function  $g: S \rightarrow \theta$ such that for every function $f:\kappa\rightarrow\mu$,
the set $\{\delta\in S\mid F(f\restriction A_\delta)= g(\delta)\}$ is stationary in $\kappa$.

For every $\delta\in S$, let $f_\delta:=h_{g(\delta)}\circ \psi_\delta$.
We claim that $\langle f_\delta\mid\delta\in S\rangle$ witnesses that $\diamondsuit(\vec L,\mu)$ holds.
Indeed, given $f: \kappa \rightarrow \mu$, consider the stationary set $S':=\{\delta\in S\mid F(f\restriction A_\delta)= g(\delta)\}$.
For every $\delta\in S'$, it is the case that $$f_\delta=h_{F(f\restriction A_\delta)}\circ \psi_\delta=f\restriction A_\delta,$$
as sought.
\end{proof}

We can at this stage describe the structure of the proof of Fact~\ref{sh775} given below.
To start, in Lemma~\ref{countinglemma} we establish an instance of the principle $\diamondsuit^*(\vec L, \ldots)$.
The caveat here is that the second parameter, the width of the diamond sequence, will be rather large.
Using this very wide diamond and an instance of a colouring principle which is established in Lemma~\ref{lemma114}, we will then in Lemma~\ref{lemma115} derive an instance of the principle $\Phi(\vec L, \mu, \theta)$.
Crucially for us here, the parameter $\theta$, the number of colours, will be large.
Finally, we will use Lemma~\ref{l210} to obtain a narrow diamond sequence on the ladder system.
The details are in Corollary~\ref{cor28}.

The upcoming proof of Fact~\ref{sh775} will make multiple uses of Shelah's revised GCH theorem \cite{Sh:460} that was briefly mentioned in the paper's introduction.
To state it, we shall need the following definition.
\begin{defn} For cardinals $\theta\le\lambda$:
\begin{itemize}
\item $\lambda^{[\theta]}$ stands for the least size of a subfamily $\mathcal A\s[\lambda]^{\le\theta}$
satisfying that every element of $[\lambda]^\theta$ is the union of less than $\theta$ many sets from $\mathcal A$;
\item $m(\lambda,\theta)$ stands for the least size of a subfamily $\mathcal A\s[\lambda]^{\theta}$
satisfying that for every $b\in [\lambda]^\theta$, there is an $a\in\mathcal A$ with $|a\cap b|=\theta$.
\end{itemize}
\end{defn}

Note that for $\theta$ a regular cardinal, $m(\lambda,\theta)\le\lambda^{[\theta]}$.

\begin{fact}[Shelah's RGCH, \cite{Sh:460}]\label{rgch}
For every pair $\Lambda\le\lambda$ of uncountable cardinals such that $\Lambda$ is a strong limit,
for co-boundedly many $\theta\in\reg(\Lambda)$, $\lambda^{[\theta]}=\lambda$.
\end{fact}

As a warm up, we prove the following lemma that may be extracted from the proof of \cite[Claim~1.11]{Sh:775}.
It concerns the principle $\onto(\ldots)$ that was recently introduced in the paper \cite{paper47} by the second and third authors.

\begin{lemma}\label{lemma114} Suppose that $\theta,\lambda$ are infinite cardinals such that $2^\theta\le m(\lambda,\theta)=\lambda$.

Then $\onto(\{\lambda\},[2^\lambda]^{\le\lambda},\theta)$ holds, i.e., there is a colouring $c:\lambda\times 2^\lambda\rightarrow\theta$
such that, for every $B\in[2^\lambda]^{\lambda^+}$,\footnote{This is not a typo. The second parameter of the principle $\onto$ is the ideal $J=[2^\lambda]^{\le\lambda}$, and the quantification here is over all sets $B$ that are $J$-positive, hence, the focus on $[2^\lambda]^{\lambda^+}$.}
there exists an $\alpha<\lambda$ such that $c[\{\alpha\}\times B]=\theta$.

\end{lemma}
\begin{proof} Let $\mathcal A$ be a witness for $m(\lambda,\theta)=\lambda$. The next claim is of independent interest.
It may be proved using elementary submodels, but we give a more elementary (for the reason of avoiding elementary submodels) proof due to Ido Feldman.
\begin{claim} \label{invclaim}For every $\mathcal H\s {}^\lambda2$ of size $\lambda^+$, there exists an $a\in\mathcal A$
such that the set $\{h \restriction a \mid h \in \mathcal H\}$ has size at least $\theta$.
\end{claim}
\begin{why}[Proof (Feldman)]  For two distinct functions $g,h\in {}^\lambda2$,
denote
$$\Delta(g,h):=\min\{\delta<\lambda\mid g(\delta)\neq h(\delta)\}.$$
Now, let a family $\mathcal H\s {}^\lambda2$ of size $\lambda^+$ be given.

$\br$ If there exists a function $g:\lambda\rightarrow2$ such that $D(g):=\{ \Delta(g,h)\mid h\in\mathcal H\setminus\{g\}\}$ has size greater than or equal to $\theta$,
then pick $a\in\mathcal A$ such that $|a\cap D(g)|=\theta$,
and for each $\delta\in a\cap D(g)$, pick some $h_\delta\in\mathcal H$ such that $\Delta(g,h_\delta)=\delta$.
Then $\delta\mapsto h_\delta\restriction a$ is injective over $a\cap D(g)$.

$\br$ Otherwise, for every $g:\lambda\rightarrow2$, let $\pi_g:\otp(D(g))\rightarrow D(g)$ be the increasing enumeration of $D(g)$,
so that $\bar g:=g\circ \pi_g$ is an element of ${}^{<\theta}2$.
As $2^{<\theta}\le2^\theta\le\lambda$, we may find $g\neq h$ in $\mathcal H$ such that $\bar g=\bar h$.
Consider $\delta:=\Delta(g,h)$. Then $\delta\in D(g)\cap D(h)$.
In addition, since $g\restriction\delta=h\restriction\delta$, $D(g)\cap\delta=D(h)\cap\delta$.
In particular, for $\xi:=\otp(D(g)\cap\delta)$, we get that $\pi_g(\xi)=\delta=\pi_h(\xi)$
and hence $g(\delta)=\bar g(\xi)=\bar h(\xi)=h(\delta)$, contradicting the definition of $\delta$.
So this case does not exist.
\end{why}

For each $a \in\mathcal A$, let $\mathcal G^a$ be the collection of all functions $g:{}^a\theta \rightarrow \theta$ such that
$$|\{ f\in{}^a\theta \mid g(f)\neq 0\}| \leq \theta.$$
Clearly, $|\mathcal G^a|=2^\theta\le\lambda$.
For each $g\in\mathcal G^a$, we lift $g$ to a function $\hat g:{}^\lambda\theta\rightarrow\theta$ by letting
$$\hat g(h):= g(h\restriction a).$$

Now, let $\langle g_\alpha \mid \alpha< \lambda\rangle$ be an injective enumeration of
$\{\hat g\mid a \in \mathcal A,\, g \in\mathcal G^a\}$,
and let $\langle h_\beta\mid \beta<2^\lambda\rangle$ be an injective enumeration of ${}^\lambda2$.
Define a colouring $c:\lambda\times 2^\lambda\rightarrow\theta$ via $c(\alpha,\beta):=g_\alpha(h_\beta)$.
To see that $c$ as as sought, let $B\in[2^\lambda]^{\lambda^+}$.
By Claim~\ref{invclaim}, pick $a\in\mathcal A$ such that $\{ h_\beta\restriction a\mid \beta\in B\}$ has size at least $\theta$.
Pick $B'\s B$ of ordertype $\theta$ on which $\beta\mapsto h_\beta\restriction a$ is injective.
It follows that we may define a function $g:{}^a\theta\rightarrow\theta$ in $\mathcal G^a$ via
$$g(f):=\begin{cases}\otp(B'\cap\beta),&\text{if }\beta \in B' \text{ and }f= h_\beta\restriction a;\\
0,&\text{otherwise}.\end{cases}$$

Pick $\alpha<\lambda$ such that $\hat g=g_\alpha$. Then, $c[\{\alpha\}\times B']=\theta$.
\end{proof}

Our next step is proving the following lemma that is easily extracted from the beginning of the proof of \cite[Claim~1.10]{Sh:775}.
It uses two fine approximations in $\zfc$ of G\"odel's constructible universe: the revised $\gch$ theorem and the approachability ideal $I[\kappa]$.

\begin{lemma} \label{countinglemma}
Suppose that $\Lambda\le\lambda$ is a pair of uncountable cardinals such that $\Lambda$ is a strong limit.
Denote $\kappa := \cf(2^\lambda)$.
Then, for co-boundedly many $\mu \in \reg(\Lambda)$, there is a $\mu$-bounded $C$-sequence $\vec C$ over some stationary $S\s E^\kappa_\mu$ such that $\diamondsuit^*(\vec C, 2^\lambda, 2^\lambda)$ holds.
\end{lemma}
\begin{proof}
We start with the following claim which guides our choice of $\mu$.
\begin{claim}\label{c261}
There is a co-bounded set of $\mu \in \reg(\Lambda)$ such that for every cardinal $\varkappa < 2^\lambda$, $\varkappa^{[\mu]} < 2^\lambda$.
\end{claim}
\begin{why}
By Fact~\ref{rgch}, for every $\varkappa\in[\Lambda,2^\lambda)$, there is a cardinal $\epsilon_\varkappa < \Lambda$ such that for every $\mu\in\reg(\Lambda)\setminus \epsilon_\varkappa$,  $\varkappa^{[\mu]}=\varkappa$. Now, as $\cf(2^\lambda)>\lambda\ge\Lambda$,
it follows that there is $\Gamma$ an unbounded subset of cardinals in $2^\lambda$ and a cardinal $\epsilon< \Lambda$ such that for every $\varkappa \in \Gamma$, $\epsilon_\varkappa < \epsilon$. In particular, for every $\mu \in \reg(\Lambda) \setminus \epsilon$, for every $\varkappa \in \Gamma$, $\varkappa^{[\mu]} = \varkappa < 2^\lambda$. This, combined with the observation that for any cardinals $\varkappa_0< \varkappa_1$ and cardinal $\mu$, $\varkappa_0^{[\mu]} \leq \varkappa_1^{[\mu]}$, verifies the claim.
\end{why}

Let $\mu$ be any cardinal in the co-bounded subset of $\reg(\Lambda)$ given by the claim.
Hereafter, all we shall need to assume about $\mu$ is that it is a regular cardinal smaller than $\lambda$ and $m(\varkappa,\mu)<2^\lambda$ for all $\varkappa<2^\lambda$.
As $\mu^+\le\lambda<\kappa$, by \cite[Claim~1.2 and Lemma~1.4]{Sh:420}, there exists a stationary $S\s E^\kappa_\mu$ that lies in $I[\kappa]$.
By possibly intersecting $S$ with some club, this means that there exists a $\mu$-bounded $C$-sequence $\vec C=\langle C_\delta\mid\delta\in S\rangle$
satisfying the following weak coherence property: for every pair $\gamma<\delta$ of ordinals from $S$,
for all $\beta\in\nacc(C_{\gamma})\cap\nacc(C_\delta)$, $C_{\gamma}\cap\beta=C_\delta\cap\beta$.
We shall prove that $\diamondsuit^*(\vec C, 2^\lambda, 2^\lambda)$ holds.

To this end, we fix the following objects.
\begin{enumerate}[(i)]
\item Let $h: \kappa \rightarrow 2^\lambda$ be increasing and cofinal.
\item Let $\mathcal C:= \{C_\delta \cap \beta \mid \delta \in S, \beta\in \nacc(C_\delta)\}$, so that $|\mathcal C| =\kappa$ and it consists of sets of size less than $\mu$.
\item Let $T:= \{f \mid \exists C \in \mathcal C\, [f \in {}^C(2^\lambda)]\}$. For each $C \in \mathcal C$, since $|C|<\mu<\lambda$,
it is the case that $2^\lambda \leq |{}^C(2^\lambda)| \leq (2^\lambda)^{|C|} = 2^\lambda$. As $|\mathcal C| = \kappa$, we conclude that $|T| =2^\lambda$.
\item Let $\langle f_i \mid i< 2^\lambda\rangle$ be an enumeration of $T$.
\item For $\delta< \kappa$, denote $T_{< \delta}:= \{f_i \mid i < h(\delta)\}$, so that $|T_{<\delta}|<2^\lambda$.
\end{enumerate}

For every $\delta \in S$, let
$$\mathcal P_\delta:= \{f \in {}^{C_\delta}h(\delta) \mid \forall \beta \in \nacc(C_\delta)\, [f \restriction (C_\delta \cap \beta) \in T_{< \delta}]\}.$$
We shall show that the sequence $\langle \mathcal P_\delta \mid \delta \in S\rangle$ is a witness for $\diamondsuit^*(\vec C, 2^\lambda, 2^\lambda)$.
We begin by estimating its width.
Implicit in the upcoming proof are the `tree powers' from the hypotheses of \cite[Claim 1.10]{Sh:775}.

\begin{claim}\label{treeclaim}
Let $\delta \in S$. Then $|\mathcal P_\delta| < 2^\lambda$.
\end{claim}
\begin{why}
Consider the set $Q_\delta:= \{g \in T_{<\delta} \mid \exists \beta \in \nacc(C_\delta)\,[g \in {}^{C_\delta \cap \beta}h(\delta)]\}$.
For every $f\in\mathcal P_\delta$,
$b_f:=\{q \in Q_\delta\mid q \s f\}$ is nothing but $\{ f\restriction\beta\mid \beta\in\nacc(C_\delta)\}$.
So from $\otp(C_\delta)=\cf(\delta)=\mu$,
we infer that $(b_f,{\s})$
is order-isomorphic to $(\mu,{\in})$, satisfying that $\bigcup b=f$ for every $b\in[b_f]^\mu$.
In particular, $|b_f\cap b_{f'}|<\mu$ for all $f\neq f'$ from $\mathcal P_\delta$.

Set $\varkappa:=|Q_\delta|$. As $\varkappa\le|T_{<\delta}|<2^\lambda$, the choice of $\mu$ ensures that $m(\varkappa,\mu)<2^\lambda$.
In particular, we may fix a subfamily $\mathcal A_\delta\s[Q_\delta]^\mu$ of size less than $2^\lambda$ such that, for every $f\in\mathcal P_\delta$, there exists $a_f\in\mathcal A_\delta$ with $|a_f\cap b_f|=\mu$. Then $f\mapsto a_f$ forms an injection from $\mathcal P_\delta$ to $\mathcal A_\delta$,
so that $|\mathcal P_\delta|<2^\lambda$.
\end{why}

We are left with verifying that $\langle \mathcal P_\delta \mid \delta \in S\rangle$ has the required guessing property. So let $f: \kappa\rightarrow{}2^\lambda$.
Let $D\s \kappa$ be a club such that $\delta \in D$ implies that
\begin{enumerate}[(i)]
\item for every $\beta< \delta$, $f(\beta) < h(\delta)$;
\item for every $\beta< \delta$, for every $\delta'\in S$ such that $\beta \in \nacc(C_{\delta'})$, we have $f \restriction (C_{\delta'}\cap \beta) \in T_{< \delta}$.
\end{enumerate}
Here we use the weak coherence property of $\vec C$ to ensure that the requirement in (ii) can indeed be satisfied.
Now suppose that $\delta \in D$.
Then for every $\beta \in \nacc(C_\delta)$, $f\restriction (C_\delta \cap \beta) \in T_{< \delta}$ and $\im(f \restriction (C_\delta \cap \beta)) \s h(\delta)$. So indeed $f \restriction C_\delta \in \mathcal P_\delta$.
\end{proof}

The last step is proving the next lemma that we extracted from the end of the proof of \cite[Claim~1.10]{Sh:775}.
The new ingredient here is the use of the principle $\onto(\ldots)$.
It generalizes Kunen's lemma \cite[Theorem~7.14]{kunen} that $\diamondsuit^-$ implies $\diamondsuit$ which amounts to the case $\omega_1=\kappa=\lambda^+=2^\lambda=\theta$, using \cite[Lemma~8.3(1)]{paper47}.

\begin{lemma}\label{lemma115} Suppose that:
\begin{itemize}
\item $\vec L=\langle A_\delta\mid\delta\in S\rangle$ is a ladder system over some stationary $S\s\kappa$;
\item $\diamondsuit(\vec L,2^\lambda,2^\lambda)$ holds with $\lambda<\kappa$;
\item $\onto(\{\lambda\},[2^\lambda]^{\le\lambda},\theta)$ holds.
\end{itemize}
Then $\Phi(\vec L,2^\lambda,\theta)$ holds.
\end{lemma}
\begin{proof} Fix a bijection $\psi:{}^\lambda(2^\lambda)\leftrightarrow2^\lambda$.
For every $\alpha<\lambda$, define a map $\psi_\alpha:2^\lambda\rightarrow2^\lambda$ via
$$\psi_\alpha(\xi):=\psi^{-1}(\xi)(\alpha).$$
The point is that for every function $\sigma:\lambda\rightarrow2^\lambda$ and every $\alpha<\lambda$, $$\psi_\alpha(\psi(\sigma))=\sigma(\alpha).$$

For every $x\s\kappa$, for every map $\eta:x\rightarrow2^\lambda$, for every $\alpha<\lambda$, we let $\eta^\alpha:=\psi_\alpha\circ\eta$, so $\eta_\alpha: x \rightarrow 2^\lambda$ as well.

Now, let $F:(\bigcup_{\delta\in S}{}^{A_\delta}2^\lambda)\rightarrow\theta$ be given. Let $\langle\mathcal P_\delta\mid\delta\in S\rangle$ be a witness for $\diamondsuit(\vec L,2^\lambda,2^\lambda)$.
Without loss of generality, we may assume that for every $\delta\in S$, each element of $\mathcal P_\delta$ is a function from $A_\delta$ to $2^\lambda$.
In particular, for all $\delta\in S$, $\eta\in\mathcal P_\delta$, and $\alpha<\lambda$, $\eta^\alpha$ is a map from $A_\delta$ to $2^\lambda$
so that $F(\eta^\alpha)$ is a well-defined ordinal
less than $\theta$.
In other words, for all $\delta\in S$ and $\eta\in\mathcal P_\delta$,
$$h_\eta:=\langle F(\eta^\alpha)\mid \alpha<\lambda\rangle$$
is a map from $\lambda$ to $\theta$.
\begin{claim} Let $\delta\in S$. There exists a function $g_\delta:\lambda\rightarrow\theta$ such that,
for every $\eta\in\mathcal P_\delta$,
there exists an $\alpha<\lambda$ with $h_\eta(\alpha)= g_\delta(\alpha)$.
\end{claim}
\begin{why} Fix a witness $c:\lambda\times2^\lambda\rightarrow\theta$ to $\onto(\{\lambda\},[2^\lambda]^{\le\lambda},\theta)$.
Notice that for every $\eta\in\mathcal P_\delta$,
the set $B_\eta:=\{\beta<2^\lambda\mid \forall\alpha<\lambda\,[h_\eta(\alpha)\neq c(\alpha,\beta)]\}$ has size no more than $\lambda$,
since otherwise, by the choice of $c$, we may pick an $\alpha<\lambda$ such that $c[\{\alpha\}\times B_\eta]=\theta$,
and in particular, $h_\eta(\alpha) \in c[\{\alpha\}\times B_\eta]$.
Now, as $|\mathcal P_\delta|<2^\lambda$, it follows that we may pick $\beta\in 2^\lambda\setminus\bigcup_{\eta\in\mathcal P_\delta}B_\eta$.
Define $g_\delta:\lambda\rightarrow\theta$ via $g_\delta(\alpha):=c(\alpha,\beta)$. Then $g_\delta$ is as sought.
\end{why}
Switching the roles of $\delta$ and $\alpha$ in the preceding claim,
we may fix a sequence $\langle g_\alpha:S\rightarrow\theta\mid \alpha<\lambda\rangle$ such that,
for every $\delta\in S$, for every $\eta\in\mathcal P_\delta$,
there exists an $\alpha<\lambda$ with $h_\eta(\alpha)= g_\alpha(\delta)$.
\begin{claim}  There exists an $\alpha<\lambda$ such that for every function $f:\kappa\rightarrow2^\lambda$, the following set is stationary in $\kappa$:
$$\{\delta\in S\mid F(f\restriction A_\delta)=g_\alpha(\delta)\}.$$
\end{claim}
\begin{why} Suppose not. Then, for every $\alpha<\lambda$, we may fix a function $f_\alpha:\kappa\rightarrow2^\lambda$ and a club $D_\alpha\s\kappa$ disjoint from
$\{\delta\in S\mid F(f_\alpha\restriction A_\delta)=g_\alpha(\delta)\}$. Define a map $\eta:\kappa\rightarrow2^\lambda$ via:
$$\eta(\gamma):=\psi(\langle f_i(\gamma)\mid i<\lambda\rangle).$$

Now, as $\langle\mathcal P_\delta\mid\delta\in S\rangle$ witnesses $\diamondsuit(\vec L,2^\lambda,2^\lambda)$,
we may pick $\delta\in\bigcap_{\alpha<\lambda}D_\alpha\cap S$
such that $\bar\eta:=\eta\restriction A_\delta$ is in $\mathcal P_\delta$.
For every $\alpha<\lambda$, recall that $\bar\eta^\alpha$ is defined as $\psi_\alpha\circ\bar\eta$, so that, for every $\gamma\in A_\delta$,
$$\bar\eta^\alpha(\gamma)=\psi_\alpha(\psi(\langle f_i(\gamma)\mid i<\lambda\rangle))=f_\alpha(\gamma).$$
That is, for every $\alpha<\lambda$, $\bar\eta^\alpha=f_\alpha\restriction A_\delta$.
As $\bar\eta\in\mathcal P_\delta$, we may pick some $\alpha<\lambda$ such that $h_{\bar\eta}(\alpha)=g_\alpha(\delta)$.
But, by the definition of $h_{\bar\eta}$, $$h_{\bar\eta}(\alpha)=F(\bar\eta^\alpha).$$ Altogether,
$$F(f_\alpha\restriction A_\delta)=F(\bar\eta^\alpha)=h_{\bar\eta}(\alpha)=g_\alpha(\delta),$$
contradicting the fact that $\delta\in D_\alpha\cap S$.
\end{why}
This completes the proof.
\end{proof}
\begin{remark} Straight-forward adjustments to the preceding proof establish the following.
Suppose that $\diamondsuit(\vec L,\mu,2^\lambda)$  holds for a given ladder system $\vec L$ over a stationary subset of $\kappa$,
and a given cardinal $\lambda<\kappa$.
For every cardinal $\theta$, if $\onto(\{\lambda\},J,\theta)$ holds for some $\mu$-complete proper ideal $J$ over $\mu$,
then so does $\Phi(\vec L,2^\lambda,\theta)$.
For a list of sufficient conditions for $\onto(\{\lambda\},J,\theta)$ to hold, see the appendix of \cite{paper53}.

On another front, note that the principles $\diamondsuit(\vec L, \mu, \theta)$ and $\Phi(\vec L, \mu, \theta)$ can be strengthened by adding an extra parameter $I$, an ideal on $\kappa$ extending $\ns_\kappa \restriction S$. In each case, the set of good guesses $\delta$ is now required to be a set in $I^+$ instead of merely a stationary subset of $S$. We leave to the interested reader to verify that most of the results in this section hold for these strengthenings for $I$ any $\lambda^+$-complete ideal on $\kappa$ extending $\ns_\kappa \restriction S$. The only change that needs to be made is that the second hypothesis of Lemma~\ref{lemma115} will now require $\diamondsuit^*(\vec L, 2^\lambda, 2^\lambda)$ instead of $\diamondsuit(\vec L, 2^\lambda, 2^\lambda)$,
which is what Lemma~\ref{lemma114} produces anyway.
\end{remark}
We are now in a condition to prove Fact~\ref{sh775}.

\begin{cor}\label{cor28}
Suppose that $\Lambda\le\lambda$ is a pair of uncountable cardinals such that $\Lambda$ is a strong limit.
Denote $\kappa := \cf(2^\lambda)$.
Then, for co-boundedly many regular cardinals $\mu<\Lambda$,
there exists a $\mu$-bounded $C$-sequence $\vec C$
over a stationary $S\s E^\kappa_\mu$ such that $\diamondsuit(\vec C\restriction S', \mu)$ holds for every stationary $S'\s S$.
\end{cor}
\begin{proof}  By Lemma~\ref{countinglemma}, for co-boundedly many $\mu \in \reg(\Lambda)$, there is a $\mu$-bounded $C$-sequence $\vec C$ over some stationary $S\s E^\kappa_\mu$ such that $\diamondsuit^*(\vec C, 2^\lambda, 2^\lambda)$ holds.
In particular, $\diamondsuit(\vec C\restriction S', 2^\lambda, 2^\lambda)$ holds for any given stationary $S'\s S$.
By Fact~\ref{rgch}, we may fix $\theta\in\reg(\Lambda)$ above $2^\mu$ such that $\lambda^{[\theta]}=\lambda$.
As $\Lambda$ is a strong limit, $2^\theta<\Lambda\le\lambda$.
Thus, by Lemma~\ref{lemma114}, $\onto(\{\lambda\},[2^\lambda]^{\le\lambda},\theta)$ holds.
Now, let $S'\s S$ be stationary. By Lemma~\ref{lemma115}, $\Phi(\vec C\restriction S',2^\lambda,\theta)$ holds.
In particular, $\Phi(\vec C\restriction S',\mu,2^\mu)$ holds.
Then, by Lemma~\ref{l210}, $\diamondsuit(\vec C\restriction S', \mu)$ holds.
\end{proof}

Before concluding this section, we would like to briefly describe an additional configuration which provide narrow diamonds over a ladder system. Above, we used Shelah's beautiful revised GCH theorem, Fact~\ref{rgch}, to obtain instances of cardinals $\theta \leq \lambda$ such that $m(\lambda, \theta) = \lambda$. The following folklore fact provides other instances (see the proof of \cite[Lemma~5.12]{MR1086455}).
\begin{fact}\label{meetingfact}
For all infinite cardinals $\theta\le\lambda < \theta^{+\cf(\theta)}$,
$m(\lambda, \theta) = \lambda$ holds.
\end{fact}

Using Fact~\ref{meetingfact} we can trace through the proofs of this section to obtain the following theorem. The reader may first consider Corollary~\ref{cor212} below
which deals with the simplest case of the theorem, where $\mu:= \aleph_0$, in which case $\mu^{+\mu}= \aleph_\omega$.
\begin{thm}\label{thm211}
Suppose that $\mu$ is an infinite regular cardinal, and $\lambda$ is a cardinal such that
$\mu<2^\mu<2^{2^\mu}\le\lambda<2^\lambda<\mu^{+\mu}$.
Then $\kappa:= 2^\lambda$ is a regular cardinal admitting a $\mu$-bounded $C$-sequence $\vec C$
over a stationary $S\s E^\kappa_\mu$ such that $\diamondsuit(\vec C\restriction S', \mu)$ holds for every stationary $S'\s S$.
\end{thm}
\begin{proof}  Let $\alpha<\mu$ be such that $\kappa=\mu^{+\alpha}$. If $\kappa$ were to be singular,
then $\cf(2^\lambda)=\cf(\alpha)<\mu<\lambda$, contradicting Konig's lemma.
So $\kappa$ is regular.
Next, as $2^\lambda<\mu^{+\mu}$, for every cardinal $\varkappa\in[\mu,2^\lambda)$, it is the case that $\mu\le\varkappa<\mu^{+\mu}$,
and so Fact~\ref{meetingfact} implies that $m(\varkappa, \mu) = \varkappa<2^\lambda$.
As made clear right after Claim~\ref{c261},
we then get a $\mu$-bounded $C$-sequence $\vec C$ over some stationary $S\s E^\kappa_\mu$ such that $\diamondsuit^*(\vec C, 2^\lambda, 2^\lambda)$ holds.
In particular, $\diamondsuit(\vec C\restriction S', 2^\lambda, 2^\lambda)$ holds for any given stationary $S'\s S$.
Now let $\theta:= 2^\mu$.
As $\mu<\theta<\lambda<\mu^{+\mu}$, it is the case that $\theta<\lambda<\theta^{+\theta}$,
so Fact~\ref{meetingfact} implies that $m(\lambda,\theta)=\lambda$.
Then, by Lemma~\ref{lemma114}, $\onto(\{\lambda\},[2^\lambda]^{\le\lambda},\theta)$ holds.
Then, by Lemma~\ref{lemma115}, $\Phi(\vec C\restriction S',2^\lambda,\theta)$ holds.
In particular, $\Phi(\vec C,\mu,2^\mu)$ holds.
Then, by Lemma~\ref{l210}, $\diamondsuit( \vec C\restriction S', \mu)$ holds.
\end{proof}

\begin{cor}\label{cor212}
For every $n\in[3,\omega)$ such that $\beth_n<\aleph_\omega$,
there is an $\omega$-bounded $C$-sequence $\vec C$ over $E^{\beth_n}_\omega$ such that $\diamondsuit(\vec C, \omega)$ holds.
In particular, if $\kappa:=2^{2^{2^{\aleph_0}}}$ is smaller than $\aleph_\omega$,
then there is an $\omega$-bounded $C$-sequence $\vec C$ over $E^\kappa_\omega$ such that $\diamondsuit(\vec C, \omega)$ holds.
\qed
\end{cor}

\begin{remark}\label{steppingupremark}
With a bit more work, one can show that if $\aleph_\omega$ is a strong limit, then for every uncountable cardinal $\mu<\aleph_\omega$,
there is a finite set $\Theta\s\aleph_\omega$ such that for every cardinal $\lambda$ with $\cf(\lambda)\notin\Theta$,
there exists a $\mu$-bounded $C$-sequence $\vec C$ over $E^{\lambda^+}_\mu$
such that $\diamondsuit(\vec C, \mu)$ holds.
This is obtained by developing stepping up methods which allow for transferring diamonds on ladder systems from smaller cardinals to larger cardinals.
\end{remark}
Having discussed the methods from the preceding remark and its limitations with Jing Zhang, the following question emerged:
\begin{q}[Zhang] \label{jingquestion}Suppose that $\aleph_\omega$ is a strong limit and $\square_{\aleph_\omega}$ holds.
Does there exist a cardinal $\mu<\aleph_\omega$,
and a $\mu$-bounded ladder system $\vec L$ over a \emph{nonreflecting} stationary subset of $E^{\aleph_{\omega+1}}_\mu$
such that $\diamondsuit(\vec L, \mu)$ holds?
\end{q}

\section{Club guessing with diamonds}\label{sec3}
The main result of this section concerns the following $n$-dimensional version of Definition~\ref{defdio}.
The main technical result, Theorem~\ref{thm32}, additionally incorporates club guessing into the ladder system on which diamond holds.

\begin{defn}\label{defdioN} For a ladder system $\vec L=\langle A_\delta\mid\delta\in S\rangle$ over some stationary $S\s\kappa$
a cardinal $\theta$ and a positive integer $n$,
$\diamondsuit^{n}(\vec L,\theta)$ asserts the existence of a sequence $\langle f_\delta\mid\delta\in S\rangle$ such that:
\begin{itemize}
\item for every $\delta\in S$, $f_\delta$ is a function from $[A_\delta]^n$ to $\theta$;
\item for every function $f:[\kappa]^n\rightarrow\theta$,
there are stationarily many $\delta\in S$ such that $f\restriction [A_\delta]^n=f_\delta$.
\end{itemize}
\end{defn}
\begin{remark}
We may also consider the variation of $\diamondsuit^{n}(\vec L,\theta)$ where we require above that for $\delta \in S$, $f_\delta:{}^{n}A_\delta \rightarrow \theta$ and these serve to guess a function $f:{}^n\kappa \rightarrow \theta$. In case $\theta$ is infinite however the two are easily seen to be equivalent.
\end{remark}

The main corollary to the results of this section is the following, which proves Theorem~\ref{thmd}.

\begin{cor} \label{cor33} Suppose that $\aleph_\omega$ is a strong limit.
For every positive integer $n$,
for all infinite cardinals $\mu\le\theta<\aleph_\omega$,
there are a cardinal $\kappa<\aleph_\omega$,
and a $\mu$-bounded ladder system $\vec L$ over $E^\kappa_\mu$ such that $\diamondsuit^n(\vec L,\theta)$ holds
and is moreover witnessed by a sequence $\langle f_\delta\mid\delta\in S\rangle$ consisting of constant maps.
\end{cor}
\begin{proof}
By Theorem~\ref{thm32} below, taking $\Omega:=\omega$.
\end{proof}

Towards the proof of Theorem~\ref{thm32}, we shall need the following strong variation of Lemma~\ref{lemma114}.
\begin{lemma}\label{ontothird} Suppose that $\theta,\lambda$ are infinite cardinals such that $\lambda^\theta=\lambda$.
Then there is a proper $\theta^+$-complete ideal $I$ over $\lambda$ such that $\onto(I^+,[2^\lambda]^{<\theta},\theta)$ holds.
That is, there is a colouring $c:\lambda\times 2^\lambda\rightarrow\theta$
satisfying that for all $A\in I^+$ and $B\in[2^\lambda]^\theta$,
there exists an $\alpha\in A$ such that $c[\{\alpha\}\times B]=\theta$.
\end{lemma}
\begin{proof} As $\lambda^\theta=\lambda$,
by the Engelking-Karlowicz theorem,  \cite{MR196693},
we may fix a sequence $\vec f=\langle f_\alpha\mid \alpha<\lambda\rangle$
of functions from $2^\lambda$ to $\lambda$
such that for every function $g:x\rightarrow\lambda$ with $x\in[2^\lambda]^\theta$, there exists an $\alpha<\lambda$ such that $g\s f_\alpha$. Let us identify a useful feature of $\vec f$.
\begin{claim}\label{claimidealproper}
For every sequence $\langle B_i \mid i< \theta\rangle$ of elements of $[2^\lambda]^\theta$, for some $\alpha< \lambda$, for every $i< \theta$, $f_\alpha[B_i] = \theta$.
\end{claim}
\begin{why}
Let $\langle B_i \mid i< \theta\rangle$ a sequence of elements of $[2^\lambda]^\theta$ be given. First, let $\langle B^*_i \mid i< \theta\rangle$ be a sequence of pairwise disjoint sets such that for each $i< \theta$, $B^*_i \in [B_i]^\theta$. Then let $x:= \bigcup_{i< \theta}B^*_i$, and let $g:x \rightarrow \lambda$ be such that for each $\beta \in x$, $g(\beta) = \xi$ iff for some $i< \theta$, $\beta \in B^*_i$ and $\otp(B^*_i \cap \beta) = \xi$. Then if $\alpha < \lambda$ is such that $g \s f_\alpha$, then for every $i< \theta$ we have that $f_\alpha[B_i] = \theta$.
\end{why}
Now let $I$ denote the collection of all $A\s\lambda$ for which there exists $\mathcal B\in[[2^\lambda]^\theta]^\theta$
such that, for every $\alpha\in A$, for some $B\in\mathcal B$ it is the case that $f_\alpha[B]\neq\theta$.
It is clear that $I$ is a $\theta^+$-complete ideal over $\lambda$ and it is proper by Claim~\ref{claimidealproper}.
Now pick $c:\lambda\times2^\lambda\rightarrow\theta$ such that $c(\alpha,\beta)=f_\alpha(\beta)$ whenever $f_\alpha(\beta)<\theta$.
Then $c$ and $I$ are as sought.
\end{proof}

The next theorem deals with getting a witness for $\diamondsuit^n(\vec L,\theta)$ from Definition~\ref{defdioN} with several additional features. First,
the local functions $\langle f_\delta \mid \delta \in S\rangle$ witnessing $\diamondsuit^n(\vec L,\theta)$ are the simplest possible: they are constant maps --- see the function $g$ in Clause~(1) below.
Second, the sequence $\langle S_j\mid j<\kappa\rangle$ of Clause~(2) shows that we have $\kappa$-many disjoint stationary sets each of which carries the desired diamond.
This is motivated by results such as \cite[Theorem~A.1]{paper54} that uses a guessing principle with $\mu$ many pairwise disjoint active parts to construct $2^\mu$ many pairwise nonhomeomorphic Dowker spaces.
Unlike the usual diamond and some of its variants (see for example \cite[Theorem~3.7]{paper23} and \cite[Lemma~3.19]{paper40}) that abstractly admit a partition into $\kappa$ many active parts,
here we do not know of such a partition theorem, hence the explicit inclusion of Clause~(2).
One possible explanation for the lack of a partition theorem is that for cardinals $\mu<\theta<\kappa$,
the collection corresponding to the failure of diamond on $\mu$-bounded ladder systems for $\theta$-colourings does not form a $\kappa$-complete ideal which prevents the use of standard non-saturation results such as Ulam's theorem and its generalisations \cite{paper47}.

The third feature of the next theorem is motivated by the study of \emph{relative club guessing}
and can be most concisely expressed as $\CG_\mu(E^\kappa_\mu,T,\kappa)$ in the sense of \cite[Definition~2.2]{paper46}.
To make our explanation self-contained, notice that the $\delta \in S_j$ below not only guesses the global function $f$ on the set $[B_\delta]^n$, but additionally, given a club $D\s \kappa$ we ensure that $\delta$ simultaneously guesses the club $D$ \emph{relative to }$T$, that is, $B_\delta \s D \cap T$. For more on the utility of this feature, we refer the reader to \cite{paper46}.

When reading the next theorem for the first time, the reader may want to ease on themselves by assuming that $n=1$ and $\Omega=\mu=\theta=\omega$, so that $\kappa=2^{2^{2^{\aleph_0}}}$.

\begin{thm}\label{thm32} Suppose that $n$ is a positive integer,
$\Omega,\mu\le\theta$ are infinite cardinals with $\mu$ regular,
$\kappa:=\beth_{n+2}(\theta)$ is smaller than $\Omega^{+\omega}$,
and $T\s\kappa$ is stationary.
Then there are:
\begin{enumerate}
\item a map $g:E^\kappa_\mu\rightarrow\theta$,
\item a partition $\langle S_j\mid j<\kappa\rangle$ of $E^\kappa_\mu$ into stationary sets, and
\item a $\mu$-bounded ladder system $\vec L=\langle B_\delta\mid\delta\in E^\kappa_\mu\rangle$,
\end{enumerate}
such that for every club $D\s\kappa$, for every function $f:[\kappa]^n\rightarrow\theta$, for every $j<\kappa$, there is a $\delta\in S_j$ such that $B_\delta\s D\cap T$ and
$f``[B_\delta]^n=\{g(\delta)\}$.
\end{thm}
\begin{proof}
By the Erd\H{o}s-Rado theorem,
the cardinal $\chi:=(\beth_{n-1}(\theta))^+$ satisfies $\chi\rightarrow(\mu)^n_\theta$.
Put $\sigma:=2^{<\chi}$ and
$\lambda:=2^\sigma$. Consequently, $\sigma = \beth_n(\theta)$, $\lambda = \beth_{n+1}(\theta)$, and $\kappa = \beth_{n+2}(\theta) = 2^\lambda$. Altogether,
$$\max\{\Omega, \mu\}\le\theta<\chi\le\sigma<\lambda=\lambda^\sigma<\kappa<\Omega^{+\omega}.$$
Fix an $m<\omega$ such that $\kappa=\lambda^{+m+1}$ and let $\Lambda:=\lambda^{+m}$, so that $\lambda \leq \Lambda$.
Note that Hausdorff's lemma implies that $\Lambda^\sigma=\Lambda<\kappa$.
In addition, since $\lambda^\sigma=\lambda$, Lemma~\ref{ontothird}
provides us with a $\sigma^+$-complete proper ideal $I$ over $\lambda$ such that $\onto(I^+,[\kappa]^{<\sigma},\sigma)$ holds.

\begin{claim}\label{claim351} There exists a $\chi$-bounded $C$-sequence $\langle C_\rho\mid \rho\in R\rangle$ such that:
\begin{itemize}
\item $R\s \acc(\kappa)\cap E^\kappa_{\le\chi}$ is stationary;
\item for every $\rho\in R$, for every $\delta\in\acc(C_\rho)$, $\delta\in R$ and $C_\delta=C_\rho\cap\delta$;
\item for every club $D\s\kappa$, for every $\varepsilon<\kappa$,
there exists a $\rho\in R\cap E^\kappa_{\chi}$ such that $\nacc(C_\rho)\s D\cap T$ and $\min(C_\rho)\ge\varepsilon$.
\end{itemize}
\end{claim}
\begin{why} The proof is similar to that of \cite[Lemma~2]{Sh:237e}.
As $\kappa=\Lambda^+$ and $\Lambda^\chi=\Lambda$,
for every $\rho\in \acc(\kappa)\cap E^\kappa_{\le\chi}$ we may let $\langle C_{\rho,j}\mid j<\Lambda\rangle$ be an enumeration of all clubs in $\rho$ of order-type no more than $\chi$.
In addition, using $\Lambda^\chi=\Lambda$, by the Engelking-Karlowicz theorem,
we may fix a sequence $\vec f=\langle f_i\mid i<\Lambda\rangle$
of functions from $\kappa$ to $\Lambda$
such that for every function $g:x\rightarrow\Lambda$ with $x\in[\kappa]^\chi$, there exists an $i<\Lambda$ such that $g\s f_i$.
Denote $C_\rho^i:=C_{\rho,f_i(\rho)}$.
Clearly, for every $i<\Lambda$, $\vec C^i:=\langle C_\rho^i\mid \rho\in\acc(\kappa)\cap E^\kappa_{\le\chi}\rangle$ is a $\chi$-bounded $C$-sequence.

We claim that there exists an $i<\Lambda$ such that for every club $D\s\kappa$, for every $\varepsilon<\kappa$,
there exists a $\rho\in R\cap E^\kappa_{\chi}$ such that:
\begin{enumerate}
\item $\nacc(C_\rho^i)\s D\cap T$;
\item $\min(C^i_\rho)\ge\varepsilon$
\item for every $\delta\in\acc(C^i_\rho)$, $C^i_\delta=C^i_\rho\cap\delta$.
\end{enumerate}
Indeed, otherwise, for each $i<\Lambda$, we may fix a club $D_i\s\kappa$ and some $\varepsilon_i<\kappa$ such that for every $\rho\in R\cap E^\kappa_\chi$,
either (1) fails for $D_i$ or (2) fails for $\varepsilon_i$ or (3) fails.
Let $D:=\bigcap_{i<\Lambda}D_i$ and $\varepsilon:=\sup_{i<\Lambda}\varepsilon_i$.
As $T$ is stationary, we may now fix some $\rho\in E^\kappa_\chi$ above $\varepsilon$ such that $D\cap T$ is cofinal in $\rho$.
Fix a club $C$ in $\rho$ of order-type $\chi$ such that $\nacc(C)\s D\cap T$ and $\min(C)=\varepsilon$.
Pick a function $g:\acc(C)\cup\{\rho\}\rightarrow\Lambda$ such that $g(\delta)=j$ implies $C\cap\delta=C_{\delta,j}$.
Pick an $i<\Lambda$ such that $g\s f_i$. This implies that
$$C^i_\rho = C_{\rho, f_i(\rho)} = C_{\rho, g(\rho)} = C,$$
and for $\delta \in \acc(C)$,
$$C^i_\delta = C_{\delta, f_i(\delta)} = C_{\delta, g(\delta)} = C\cap \delta.$$
Then we arrive at the following contradiction:
\begin{enumerate}[(i)]
\item $\nacc(C_\rho^i)= \nacc(C)\s D\cap T\s D_i\cap T$;
\item $\min(C^i_\rho)=\min(C)=\varepsilon\ge\varepsilon_i$;
\item for every $\delta\in\acc(C^i_\rho)$, $C^i_\delta=C\cap\delta=C^i_\rho\cap\delta$.
\end{enumerate}

Thus, pick $i<\Lambda$ such that for every club $D\s\kappa$, for every $\varepsilon<\kappa$,
there exists a $\rho\in R\cap E^\kappa_{\chi}$ such that (1)--(3) holds.
Set
$$R:=\{ \rho\in\acc(\kappa)\cap E^\kappa_{\le\chi}\mid \forall\delta\in\acc(C^i_\rho)\,[C^i_\rho\cap\delta=C^i_\delta]\}.$$
Then $\langle C^i_\rho\mid \rho\in R\rangle$ is as sought.
\end{why}

Let $\langle C_\rho\mid \rho\in R\rangle$ be given by the claim.
It follows that the set
$$\mathcal E:=\{ \varepsilon<\kappa\mid \forall D\s\kappa\text{ club }\exists \rho\in R\cap E^\kappa_{\chi}\,[\nacc(C_\rho)\s^* D\cap T\ \&\ \min(C_\rho)=\varepsilon]\}$$
is cofinal in $\kappa$. For every $j<\kappa$, let $\mathcal E(j)$ denote the unique $\varepsilon\in\mathcal E$ to satisfy $\otp(\mathcal E\cap\varepsilon)=j$,
and then let
\begin{itemize}
\item $R_j:=\{\rho\in R\cap E^\kappa_\chi\mid \min(C_\rho)=\mathcal E(j)\}$, and
\item $S_j:=\bigcup\{ \acc(C_\rho)\mid \rho\in R_j\}\cap E^\kappa_\mu$.
\end{itemize}

Clearly, $\langle S_j\mid j<\kappa\rangle$ is a partition into stationary sets of some subset $S$ of $E^\kappa_\mu$, since for $j < j'< \kappa$ we have that $\mathcal E(j) \neq \mathcal E(j')$ and by the coherence property of $\langle C_\rho\mid \rho\in R\rangle$.
For each $\delta\in S$, let
$$\mathcal P_\delta:=\bigcup\{{}^{[C_\delta\setminus\epsilon]^n}\delta\mid \epsilon\in C_\delta\}$$
and note that
$$|\mathcal P_\delta|\le\Lambda^{\chi}\le \Lambda^\sigma= \Lambda<\Lambda^+ =\kappa.$$
For each $\delta\in S$, since $\otp(C_\delta)\le\chi$ but $\cf(\delta)=\mu<\chi$, it is the case that $\otp(C_\delta)<\chi$.
Combining this with $\theta\cdot 2^{<\chi}=\sigma$, let us also fix an enumeration $\langle (\tau_{\delta,i},B_{\delta,i})\mid i<\sigma\rangle$ of all pairs $(\tau,B)$
such that $\tau<\theta$ and $B$ is a cofinal subset of $\nacc(C_\delta)$ of ordertype $\mu$.

Moving on, as $\kappa=2^\lambda=\theta^\lambda$, let us fix a bijection $\psi:{}^\lambda\theta\leftrightarrow\kappa$.
For every $\alpha<\lambda$, define a map $\psi_\alpha:\kappa\rightarrow\theta$ via
$$\psi_\alpha(\xi):=\psi^{-1}(\xi)(\alpha).$$
For every subset $C\s\kappa$, for every map $\eta:[C]^n\rightarrow\kappa$, for every $\alpha<\lambda$, we let $\eta^\alpha:=\psi_\alpha\circ\eta$, so that $\eta^\alpha$ is a function from $[C]^n$ to $\theta$.
We say that $f:[C]^n\rightarrow\theta$ is \emph{good}
iff there is some cofinal subset $H\s\nacc(C)$ such that $f\restriction[H]^n$ is constant.

Let $\delta\in S$ and $\eta\in\mathcal P_\delta$. So, for some $\epsilon \in C_\delta$, $\eta: [C_\delta\setminus\epsilon]^n \rightarrow \delta$.
As $\im(\eta)\s\kappa$, for every $\alpha<\lambda$, $\eta^\alpha:[C_\delta\setminus\epsilon]^n\rightarrow \theta$,
so we let $A_\eta:=\{\alpha<\lambda\mid \eta^\alpha\text{ is good}\}$.
Pick a map $h_\eta:\lambda\rightarrow\sigma$ such that for every $\alpha\in A_\eta$,
$i:=h_\eta(\alpha)$ satisfies that $B_{\delta,i}$ is a cofinal subset of $\nacc(C_\delta\setminus\epsilon)$ for which $\eta^\alpha\restriction[B_{\delta,i}]^n$ is constant with value $\tau_{\delta,i}$.

\begin{claim}\label{claim1061} Let $\delta\in S$. There exists a function $g_\delta:\lambda\rightarrow\sigma$ such that,
for every $\eta\in\mathcal P_\delta$ with $A_\eta\in I^+$,
there exists an $\alpha\in A_\eta$ with $h_\eta(\alpha)= g_\delta(\alpha)$.
\end{claim}
\begin{why} Fix a colouring $c:\lambda\times\kappa\rightarrow\sigma$ witnessing $\onto(I^+,[\kappa]^{<\sigma},\sigma)$.
Let $\eta\in\mathcal P_\delta$ such that $A_\eta\in I^+$.
Note that the set $B_\eta:=\{\beta<\kappa\mid \forall\alpha\in A_\eta\,[h_\eta(\alpha)\neq c(\alpha,\beta)]\}$ has size less than $\sigma$.
Indeed, otherwise, $B_\eta\in[\kappa]^{\sigma}$,
and so since $A_\eta\in I^+$,
we may pick an $\alpha\in A_\eta$ such that $c[\{\alpha\}\times B_\eta]=\sigma$,
and in particular, $h_\eta(\alpha) \in c[\{\alpha\}\times B_\eta]$.

As $|\mathcal P_\delta|<\kappa$, it follows that we may pick $\beta\in \kappa\setminus\bigcup\{B_\eta\mid \eta\in\mathcal P_\delta\ \&\ A_\eta\in I^+\}$.
Define $g_\delta:\lambda\rightarrow\sigma$ via $g_\delta(\alpha):=c(\alpha,\beta)$. Then $g_\delta$ is as sought.
\end{why}

Let $\langle g_\delta\mid\delta\in S\rangle$ be given by the preceding claim.

\begin{claim}  Let $j<\kappa$. There exists an $\alpha<\lambda$ such that for every function $f:[\kappa]^n\rightarrow\theta$,
for every club $D\s\kappa$, there is a $\delta\in S_j$ such that $B_{\delta,g_\delta(\alpha)}\s D\cap T$ and $f``[B_{\delta,g_\delta(\alpha)}]^n=\{\tau_{\delta,g_\delta(\alpha)}\}$.
\end{claim}
\begin{why} Suppose not. Then, for every $\alpha<\lambda$, we may fix a function $f_\alpha:[\kappa]^n\rightarrow\theta$ and a club $D_\alpha\s\kappa$
such that for every $\delta\in S_j$,
either $B_{\delta,g_\delta(\alpha)}\nsubseteq D_\alpha\cap T$ or $f_\alpha``[B_{\delta,g_\delta(\alpha)}]^n\neq\{\tau_{\delta,g_\delta(\alpha)}\}$.
Using that $\psi:{}^\lambda\theta\leftrightarrow\kappa$ is a bijection define a map $\eta:[\kappa]^n\rightarrow\kappa$ via:
$$\eta(\gamma_1,\ldots,\gamma_n):=\psi(\langle f_\beta(\gamma_1,\ldots,\gamma_n)\mid \beta<\lambda\rangle).$$

Consider the club $D:=\{\delta\in\bigcap_{\alpha<\lambda}D_\alpha\mid \eta``[\delta]^n\s\delta\}$.
Pick $\rho\in R_j$ and a large enough $\epsilon\in C_\rho$ such that $\nacc(C_\rho\setminus\epsilon)\s D\cap T$.
As $\otp(\nacc(C_\rho\setminus\epsilon))=\chi$, recalling that $\chi\rightarrow(\mu)^n_\theta$ holds,
for every $\alpha<\lambda$,
we may pick a subset $H_\alpha\s\nacc(C_\rho\setminus\epsilon)$ of order-type $\mu$ such that $f_\alpha\restriction[H_\alpha]^n$ is constant,
and clearly $\delta_\alpha:=\sup(H_\alpha)$ is an element of $\acc(C_\rho\setminus \epsilon)\cap E^\kappa_\mu$, and hence of $S_j\cap D$ too.
As $|C_\rho|=\chi\le\sigma$ and $I$ is a proper $\sigma^+$-complete ideal on $\lambda$,
let us now pick a $\delta\in S_j\cap D$ such that $A^*:=\{\alpha<\lambda\mid\delta_\alpha=\delta\}$ is in $I^+$.

Since $\delta\in\acc(C_\rho)$, it is the case that $C_\delta=C_\rho\cap\delta$, and hence, for every $\alpha\in A^*$, $H_\alpha$ witnesses that $f_\alpha\restriction[C_\delta\setminus\epsilon]^n$ is good.

Since $\delta\in D$,
we know that $\bar\eta:=\eta\restriction[C_\delta\setminus\epsilon]^n$ is in $\mathcal P_\delta$.
For every $\alpha<\lambda$, recall that $\bar\eta^\alpha$ is defined as $\psi_\alpha\circ\bar\eta$, so that, for every $(\gamma_1,\ldots,\gamma_n)\in[C_\delta\setminus\epsilon]^n$,
$$\bar\eta^\alpha(\gamma_1,\ldots,\gamma_n)=\psi_\alpha(\psi(\langle f_\beta(\gamma_1,\ldots,\gamma_n)\mid \beta<\lambda\rangle))=f_\alpha(\gamma_1,\ldots,\gamma_n).$$
That is, $\bar\eta^\alpha=f_\alpha\restriction[C_\delta\setminus\epsilon]^n$ for every $\alpha<\lambda$,
so $A_{\bar\eta}$ is equal to $\{\alpha<\lambda\mid f_\alpha\restriction[C_\delta\setminus\epsilon]^n\text{ is good}\}$ and it covers the $I^+$-set $A^*$,
and hence $A_{\bar\eta}\in  I^+$.
Recalling that $g_\delta$ was given to us by Claim~\ref{claim1061}, we may now pick an $\alpha\in A_{\bar\eta}$ with $h_{\bar\eta}(\alpha)= g_\delta(\alpha)$.
By the definition of $h_{\bar\eta}$, this means that $B_{\delta,g_\delta(\alpha)}$ is a cofinal subset of $\nacc(C_\delta\setminus\epsilon)$ for which $\bar\eta^\alpha\restriction[B_{\delta,g_\delta(\alpha)}]^n$ is constant with value $\tau_{\delta,g_\delta(\alpha)}$.
But
$$B_{\delta, g_\delta(\alpha)} \s\nacc(C_\delta\setminus\epsilon)\s\nacc(C_\rho\setminus\epsilon)\s D\cap T\s D_\alpha\cap T$$ and $\bar\eta^\alpha\restriction [B_{\delta,g_\delta(\alpha)}]^n=f_\alpha\restriction [B_{\delta,g_\delta(\alpha)}]^n$,
contradicting the choice of $D_\alpha$ and $f_\alpha$.
\end{why}

For every $j<\kappa$, let $\alpha_j<\lambda$ be given by the preceding claim,
and then for every $\delta\in S_j$, let $B_\delta:=B_{\delta,g_\delta(\alpha_j)}$ and $g(\delta):=\tau_{\delta,g_\delta(\alpha_j)}$.
Then $\langle B_\delta\mid\delta\in S\rangle$, $\langle S_j\mid j<\kappa\rangle$ and $g:E^\kappa_\mu\rightarrow\theta$ are as sought
modulo the fact that $S=\bigcup_{j<\kappa}S_j$ may possibly be a proper subset of $E^\kappa_\mu$, but this can be mitigated  by allocating all the left-out points to $S_0$
and defining $B_\delta$ and $g(\delta)$ arbitrarily over these points.
\end{proof}
\begin{remark}\label{remark38} An inspection of the preceding proof shows that
if $n>1$, then we may as well take $\sigma$ to be $\chi$ since in this case, $\theta\cdot\chi^\mu=\chi$. Thus, for $n>1$, the proof yields the same conclusion for $\kappa:=\beth_2((\beth_{n-1}(\theta))^+)$
(instead of $\beth_{n+2}(\theta)$) assuming it is smaller than $\Omega^{+\omega}$.
\end{remark}
\newpage
\begin{remark} In the special case $\mu=\theta$, assuming $\diamondsuit(E^\kappa_\mu)$, it is possible to arrange a $\mu$-bounded ladder system $\vec L=\langle A_\delta\mid\delta\in E^\kappa_\mu\rangle$
and a sequence $\langle f_\delta \mid \delta \in E^\kappa_\mu\rangle$ witnessing $\diamondsuit^n(\vec L, \theta)$
as in Definition~\ref{defdioN}
such that, for some `wild' function $f:[\kappa]^n\rightarrow\theta$, for every $\delta$ in the stationary set $\{\delta\in E^\kappa_\mu\mid f \restriction [A_\delta]^n = f_\delta\}$, it is the case that $f\restriction[A_\delta]^n$ is a \emph{bijection}.
For the purposes of our application of Theorem~\ref{thm32} in Corollary~\ref{cor47} however this is no good,
and we need at least a small portion of the Ramsey-theoretic feature provided by the function $g$.
The reason can be gleaned from considering the negation of Clause~(3) of Fact~\ref{thm31}.
\end{remark}
We conclude this section with pointing out that similar to Remark~\ref{steppingupremark},
by using standard stepping up methods
one can transfer Corollary~\ref{cor33} to $\aleph_{\omega+1}$, as follows.

\begin{cor} Suppose that $\aleph_\omega$ is a strong limit.
For every positive integer $n$ and every $\theta<\aleph_\omega$,
there exist a ladder system $\vec L=\langle A_\delta\mid\delta<\aleph_{\omega+1}\rangle$ and a function $g:\aleph_{\omega+1}\rightarrow\theta$
such that for every function $f:[\aleph_{\omega+1}]^n\rightarrow\theta$,
for every uncountable cardinal $\mu<\aleph_\omega$,
$\{\delta\in E^{\aleph_{\omega+1}}_\mu\mid f``[A_\delta]^n=\{g(\delta)\}\}$ is a reflecting stationary set. \qed
\end{cor}

\section{Ladder systems and topological spaces}\label{sec1}

In this section, we give the first topological application of diamonds on ladder systems to topology.
While not stated explicitly so far, we shall want the topological spaces constructed in this paper to at least be Hausdorff. Thus,
we shall need the following folklore fact.
\begin{fact}\label{fact32} For a ladder system $\vec L=\langle A_\delta\mid\delta\in S\rangle$, all of following are equivalent:
\begin{enumerate}
\item $X_{\vec L}$ is Hausdorff;
\item $X_{\vec L}$ is Hausdorff and regular;
\item for every pair $\gamma<\delta$ of ordinals from $S$, $\sup(A_\gamma\cap A_\delta)<\gamma$.
\end{enumerate}
\end{fact}

In particular, if $\vec L$ is $\omega$-bounded, then $X_{\vec L}$ is Hausdorff and regular. More generally,
for every $\mu$-bounded ladder system $\vec L$ over a subset of $E^\kappa_\mu$,  it is the case that $X_{\vec L}$ is Hausdorff and regular.
A second basic fact will be needed.
Namely, by \cite[Proposition~4.1]{leiderman2023deltaspaces}
and a straight-forward generalisation of \cite[Claim~1]{MR2099600}, we have the following characterisation.

\begin{fact}\label{thm31} For a ladder system $\vec L=\langle A_\delta\mid\delta\in S\rangle$, all of following are equivalent:
\begin{enumerate}
\item $X_{\vec L}$ is a $\Delta$-space;
\item $X_{\vec L}$ is countably metacompact;
\item For every function $g:S \rightarrow \omega$, there is a function $f: \kappa \rightarrow \omega$, such that, for every $\delta \in S$, $\sup\{\alpha\in A_\delta \mid f(\alpha) \leq g(\delta)\} <\delta$.
\end{enumerate}
\end{fact}

\begin{remark} The above characterisation makes it clear that an $\omega$-bounded ladder system $\vec L$ over $\omega_1$ for which $X_{\vec L}$ is not countably metacompact
can be constructed from a gallery of hypotheses. To mention just two, by \cite[Theorem~3.7]{paper23}
such a ladder system exists assuming $\clubsuit$,
and by the proof of \cite[Corollary~4.6]{paper54} such a ladder system exists assuming $\diamondsuit(\mathfrak b)$.
\end{remark}

\begin{lemma}\label{midapp} Suppose that $\mu<\kappa$ is a pair of regular uncountable cardinals,
and that $\vec L$ is a $\mu$-bounded ladder system over $E^\kappa_\mu$ such that $\Phi(\vec L,\omega,\omega)$ holds. Then $X_{\vec L}$ is a regular Hausdorff space that is not countably metacompact.
\end{lemma}
\begin{proof}
Since $\vec L$ is a $\mu$-bounded ladder system over $E^\kappa_\mu$,
Fact~\ref{fact32} implies that $X_{\vec L}$ is regular and Hausdorff.
Write $\vec L$ as $\langle A_\delta \mid \delta \in E^\kappa_\mu\rangle$.
Define a function $F:(\bigcup_{\delta\in S}{}^{A_\delta}\omega)\rightarrow\omega$
be letting for all $\delta\in S$ and $ f:A_\delta\rightarrow\omega$,
$$F( f):=\min\{n<\omega\mid \sup\{\alpha\in A_\delta \mid f(\alpha) =n\} =\delta\}.$$
Since $\Phi(\vec L,\omega,\omega)$ holds,
we may now fix a function $g:S\rightarrow\omega$ such that, for every function $f:\kappa\rightarrow\omega$,
there are stationarily many $\delta\in S$ such that $F(f\restriction A_\delta)= g(\delta)$.
In particular,  for every function $f:\kappa\rightarrow\omega$ there are stationarily many $\delta\in S$ such that $\sup\{\alpha\in A_\delta \mid f(\alpha) =g(\delta)\} =\delta$.
So, by Fact~\ref{thm31}, $X_{\vec L}$ is not countably metacompact.
\end{proof}

We are now ready to prove Theorem~\ref{thma}.
Indeed, it follows by taking $\Lambda=\lambda=\beth_\omega$ in the next theorem.

\begin{cor}\label{thm34}
Suppose that $\Lambda\le\lambda$ is a pair of uncountable cardinals such that $\Lambda$ is a strong limit.
Denote $\kappa := \cf(2^\lambda)$.
Then there are co-boundedly many $\mu\in\reg(\Lambda)$ such that $E^\kappa_\mu$
carries a $\mu$-bounded ladder system ${\vec L}$ such that  $X_{\vec L}$ is a regular Hausdorff space that is not countably metacompact.
\end{cor}
\begin{proof}  By Corollary~\ref{cor28},
there are co-boundedly many uncountable $\mu\in\reg(\Lambda)$,
for which there exists a $\mu$-bounded ladder system $\vec L$ over $E^\kappa_\mu$ such that $\diamondsuit(\vec L, \omega)$ holds, in particular, $\Phi(\vec L, \omega,\omega)$ holds. Now, appeal to Lemma~\ref{midapp}.
\end{proof}

\begin{thm} If $\kappa:=2^{2^{\aleph_1}}$ is smaller than $\aleph_{\omega_1}$,
then there exists an $\omega_1$-bounded
ladder system $\vec L$
over $E^\kappa_{\omega_1}$
for which $X_{\vec L}$ is a regular Hausdorff space that is not countably metacompact.
\end{thm}
\begin{proof}  Denote $\mu:=\aleph_1$ and $\lambda:=2^\mu$. Note that $\aleph_1<\lambda<\cf(\kappa)\le\kappa<\aleph_{\omega_1}$ and hence $\kappa$ is regular.
As $2^\lambda<\mu^{+\mu}$, for every cardinal $\varkappa\in[\mu,2^\lambda)$, it is the case that $\mu\le\varkappa<\mu^{+\mu}$,
and so Fact~\ref{meetingfact} implies that $m(\varkappa, \mu) = \varkappa<2^\lambda$.
As made clear right after Claim~\ref{c261},
we then get a $\mu$-bounded $C$-sequence $\vec C$ over some stationary $S\s E^\kappa_\mu$ such that $\diamondsuit^*(\vec C, 2^\lambda, 2^\lambda)$ holds.
In particular, there is a $\mu$-bounded $C$-sequence $\vec C$ over $E^\kappa_\mu$ such that $\diamondsuit(\vec C, 2^\lambda, 2^\lambda)$ holds.
As $\lambda^\mu=\lambda$, Lemma~\ref{lemma114} implies that $\onto(\{\lambda\},[2^\lambda]^{\le\lambda},\mu)$ holds.
Then, by Lemma~\ref{lemma115}, $\Phi(\vec C,2^\lambda,\mu)$ holds.
Now, appeal to Lemma~\ref{midapp}.
\end{proof}

We are also in a condition to prove Theorem~\ref{thmb}.
\begin{cor}\label{cor47}
If $\kappa:=2^{2^{2^{\aleph_0}}}$ is smaller than $\aleph_\omega$,
then there exists an $\omega$-bounded
ladder system $\vec L$
over $E^\kappa_{\omega}$
for which $X_{\vec L}$ is a regular Hausdorff space that is not countably metacompact.
\end{cor}
\begin{proof} Suppose that $\kappa:=2^{2^{2^{\aleph_0}}}$ is smaller than $\aleph_\omega$.
Appealing to Theorem~\ref{thm32} with $(n,\Omega,\mu,\theta):=(1,\omega,\omega,\omega)$,
we obtain an $\omega$-bounded ladder system $\vec L=\langle A_\delta\mid\delta\in E^\kappa_\omega\rangle$
and a map $g:\kappa\rightarrow\omega$
such that for every function $f:\kappa\rightarrow\omega$, there are stationarily many $\delta\in E^\kappa_\omega$ such that
$f`` A_\delta=\{g(\delta)\}$.
Since $\vec L$ is a $\mu$-bounded ladder system over $E^\kappa_\mu$,
Fact~\ref{fact32} implies that $X_{\vec L}$ is regular and Hausdorff.
In addition, by Fact~\ref{thm31}, $X_{\vec L}$ is not countably metacompact.
\end{proof}

We conclude this paper by providing a proof of Theorem~\ref{thmc}.
The definition of a regressive tree may be found in \cite[Definition~2.14]{paper48},
and the fact that coherent trees are regressive is easy to be seen.
\begin{thm} If there exists a $\kappa$-Souslin tree $\mathbf T$,
then there exists a ladder system $\vec L$ over some stationary subset of $\kappa$
for which $X_{\vec L}$ is a regular Hausdorff space that is not countably metacompact.
If the tree $\mathbf T$ is regressive, then $\dom(\vec L)=E^\kappa_\omega$ so that $X_{\vec L}$ is moreover first countable.
\end{thm}
\begin{proof}By \cite[Theorem~2.29]{paper48}, the existence of a $\kappa$-Souslin tree $\mathbf T$ implies that $\clubsuit_{\ad}(\mathcal S,1,1)$ holds for some $\kappa$-sized pairwise disjoint family $\mathcal S$ of stationary subsets of $\kappa$.
Denote $S:=\biguplus\mathcal S$.
By \cite[Corollary~2.25(2)]{paper48},  if $\mathbf T$ is regressive, then we may moreover secure that $S=E^\kappa_\omega$.
As $\mathcal S$ is infinite, let $\langle S_n\mid n<\omega\rangle$ be a partition of $S$ in such a way that each $S_n$ covers some set from $\mathcal S$,
so that $\clubsuit_{\ad}(\{S_n\mid n<\omega\},1,1)$ holds.
Finally, the latter means that there exists a ladder system $\vec L=\langle A_\delta\mid \delta\in S\rangle$ such that the following two hold:
\begin{enumerate}[(i)]
\item for every cofinal $A\s\kappa$, for every $n<\omega$, there exists a $\delta\in S_n$ such that $\sup(A_\delta\cap A)=\delta$;
\item for every pair $\gamma<\delta$ of ordinals from $S$, $\sup(A_{\gamma}\cap A_{\delta})<\gamma$.
\end{enumerate}

Now letting $g:S\rightarrow\omega$ describe the partition of $S$,
we get that for every function $f: \kappa \rightarrow \omega$, by picking $n<\omega$ such that $A:=f^{-1}\{n\}$ is cofinal in $\kappa$,
we may find $\delta\in S_n$ such that $\sup(A_\delta\cap A)=\delta$,
and hence $\sup\{\alpha\in A_\delta \mid f(\alpha) =g(\delta)\} =\delta$.
So Clause~(i) implies that $X_{\vec L}$ is not countably metacompact by Fact~\ref{thm31},
and Clause~(ii) ensures that $X_{\vec L}$ is a regular Hausdorff space by Fact~\ref{fact32}.
\end{proof}

\section{Acknowledgments}
We thank Ido Feldman for the combinatorial proof of Claim~\ref{invclaim}. 
We thank Jing Zhang for a discussion on this paper and in particular for isolating Question~\ref{jingquestion}.

The first author was supported by the European Research Council (grant agreement ERC-2018-StG 802756).
The second author was supported by the Israel Science Foundation (grant agreement 665/20).
The third author was partially supported by the Israel Science Foundation (grant agreement 203/22)
and by the European Research Council (grant agreement ERC-2018-StG 802756).
\newcommand{\etalchar}[1]{$^{#1}$}

\end{document}